\documentclass[11pt]{amsart}
\usepackage{amsmath,amsthm,amsfonts,amssymb,latexsym,enumerate,url,hyperref,color} 
\usepackage[shortlabels]{enumitem}
\usepackage[top=30mm,right=30mm,bottom=30mm,left=30mm]{geometry}

\theoremstyle{plain}

\newtheorem{thm}{Theorem}[section]
\newtheorem{lem}[thm]{Lemma}

\newtheorem{prop}[thm]{Proposition}

\newtheorem*{thmA}{THEOREM A}
\newtheorem*{corB}{COROLLARY B}

\theoremstyle{definition}

\numberwithin{equation}{section}

\newcommand{\Stab}{\operatorname{Stab}}
\newcommand{\Fix}{\operatorname{Fix}}
\newcommand{\Aut}{\operatorname{Aut}}

\newcommand{\Irr}{\operatorname{Irr}}
\newcommand{\SL}{\operatorname{SL}}
\newcommand{\FF}{\mathbb{F}}

\def\irrp#1{{\rm Irr}_{p'}(#1)}
\def\SmallGroup{\tt{SmallGroup}}

\def\irr#1{{\rm Irr}(#1)}

\def\cent#1#2{{\bf C}_{#1}(#2)}
\def\syl#1#2{{\rm Syl}_#1(#2)}
\def\nor{\unlhd}
\def\oh#1#2{{\bf O}_{#1}(#2)}
\def\Oh#1#2{{\bf O}^{#1}(#2)}
\def\zent#1{{\bf Z}(#1)}

\def\ker#1{{\rm ker}(#1)}
\def\norm#1#2{{\bf N}_{#1}(#2)}

\def\sbs{\subseteq}

\newcommand{\Gal}{{\mathrm {Gal}}}

\newcommand{\Sym}{{\mathrm {Sym}}}

\newcommand{\AAA}{{\sf A}}
\newcommand{\SSS}{{\sf S}}

\newcommand{\QQ}{{\mathbb Q}}

\renewcommand{\emptyset}{\varnothing}

\newcommand{\GAP}{\textsf{GAP}} 

\def\irr#1{{\rm Irr}(#1)}
\def\cent#1#2{{\bf C}_{#1}(#2)}

\def\syl#1#2{{\rm Syl}_#1(#2)}

\def\nor{\trianglelefteq\,}
\def\zent#1{{\bf Z}(#1)}
\def\sbs{\subseteq}
\def\gen#1{\langle#1\rangle}
\def\aut#1{{\rm Aut}(#1)}

\def\fit#1{{\bf F}(#1)}
\def\frat#1{{\bf \Phi}(#1)}

\def\irrp#1{{\rm Irr_{+}}(#1)}
\def\o#1{\overline{#1}}

\def\irr#1{{\rm Irr}(#1)}

\def\cent#1#2{{\bf C}_{#1}(#2)}

\def\syl#1#2{{\rm Syl}_#1(#2)}
\def\nor{\trianglelefteq\,}
\def\norm#1#2{{\bf N}_{#1}(#2)}
\def\oh#1#2{{\bf O}_{#1}(#2)}
\def\Oh#1#2{{\bf O}^{#1}(#2)}
\def\zent#1{{\bf Z}(#1)}
\def\sbs{\subseteq}
\def\gen#1{\langle#1\rangle}

\def\SL#1{{\rm SL}_{2}(#1)}

\begin{document}

\renewcommand{\thefootnote}{\fnsymbol{footnote}}
\footnotesep6.5pt

\title{Finite groups whose real irreducible representations have unique dimensions}

\author[T. Breuer]{Thomas Breuer}
\address{T. Breuer: Lehrstuhl f\"{u}r Algebra und Zahlentheorie, RWTH Aachen Universty, Germany}
\email{sam@math.rwth-aachen.de}

\author[F. Calegari]{Frank Calegari}
\address{F. Calegari: Department of Mathematics, University of Chicago, Chicago, IL 60637, USA} 
\email{fcale@uchicago.edu}

\author[S. Dolfi]{Silvio Dolfi}
\address{S. Dolfi: Dipartimento di Matematica, Universit\`a di Firenze, 50134 Firenze, Italy}
\email{silvio.dolfi@unifi.it}

\author[G. Navarro]{Gabriel Navarro}
\address{G. Navarro: Departament de Matem\`atiques, Universitat de Val\`encia, 46100 Burjassot,
Va\-l\`en\-cia, Spain}
\email{gabriel@uv.es}

\author[P. H. Tiep]{Pham Huu Tiep}
\address{P. H. Tiep: Department of Mathematics, Rutgers University, Piscataway, NJ 08854, USA}
\email{tiep@math.rutgers.edu}
  
\thanks{T. B. is supported by the German Research Foundation (DFG) -- Project-ID 286237555 --
 within the SFB-TRR 195 Symbolic Tools in Mathematics and their Applications.
 F.C. \ was supported in part by NSF Grant DMS-2001097.
S.D. was partially supported by INDAM-Gnsaga. 
 The research of G. N. is supported by Ministerio de Ciencia
e Innovaci\'on PID2022-137612NB-I00 funded by MCIN/AEI/ 10.13039/501100011033 and ERDF ``A way of making Europe."
 P. H. T. gratefully acknowledges the support of the NSF (grants DMS-1840702 and DMS-2200850),
the Simons Foundation, and the Joshua Barlaz Chair in Mathematics.  }
\thanks{The authors are grateful to the referee for a careful reading and helpful comments on the paper.}

\keywords{Characters,  Fields of Values, Degrees}

\subjclass[2010]{Primary 20D20; Secondary 20C15}

\begin{abstract}
We determine the finite groups whose inequivalent real irreducible representations have different degrees.
\end{abstract}

\maketitle

\section{Introduction} 
In this paper, we answer 
the following question: what are the finite groups
 whose real irreducible
 representations all have different degrees?
 Brauer's Problem 1 in his famous list of problems asks which finite-dimensional associative
algebras are group algebras. 
For instance, if $F$ is a subfield of the complex numbers,
 for which finite groups $G$ we have that the algebra
 $FG$ is a direct sum of pairwise non-isomorphic simple summands? 
If $F$ is the field of complex numbers, then the complex group algebra is determined
 by the character degrees and their multiplicities. This multiplicity is of course related to the field of values of the character. In the problem mentioned above, that is,  if
all the irreducible complex characters of a finite group $G$ occurs with multiplicity 1, then
all the irreducible characters of $G$ are necessarily rational valued. Using the Classification of Finite Groups
and the main result from W. Feit and G. Seitz \cite{FeitSeitz}, it is easy to see that no non-trivial finite group has this property. (This was also noticed in 
\cite[Lemma~1]{BCH}.) There does not seem to be a proof of this result
which avoids the Classification. 
In this paper, we are concerned with the real group algebra and the corresponding result for the inequivalent real irreducible representations. Perhaps somewhat surprisingly, we will show that exactly twelve finite non-trivial groups satisfy this property.

In fact, we work under a slightly more general hypothesis. 
If $G$ is a finite group, let $\irr G$  be the set of the irreducible complex characters of $G$, and recall that the Frobenius--Schur indicator of  $\chi \in \irr G$, $\nu(\chi)$,  is 1 if $\chi$ can be afforded by a representation over the real numbers; it is $-1$, if it is real-valued but cannot be afforded by a real representation; and it is $0$ if it is not real-valued.

\begin{thmA}\label{mainA}
Let $G$ be a finite group. Suppose that whenever $\alpha,\beta \in \irr G$ have the same degree and Frobenius--Schur indicator, then $\alpha$ and $\beta$ are complex-conjugate.

\begin{enumerate}[\rm(a)]

\item
If $G$ is solvable, then $G$ is a factor group of either
  $(C_2 \times C_2 \times C_2)\rtimes (C_7\rtimes C_3)$,  or 
  $(C_p\times C_p)\rtimes \SL 3$ for $p \in \{3, 5\}$.

\smallskip
\item
If $G$ is non-solvable, then $G$ is almost simple, and isomorphic to one of the groups in the following list
\begin{equation}\label{list1}
  \mathcal{L} = \{ \AAA_8, {\rm SL}_3(2),  {\rm M}_{11} , {\rm M}_{22},{\rm M}_{23}, {\rm M}_{24}, {\rm SU}_3(3), {\rm McL}, {\rm Th}, {\rm SL}_2(8).3 ,{\rm O}_8^+(2).3 \}. 
\end{equation}
\end{enumerate}
\end{thmA} 

Of these groups, ${\rm McL}$ 
is the unique group which admits real representations~$\alpha$, $\beta$
of the same dimension ($3520$) but with distinct Frobenius--Schur indicator. Furthermore, the non-trivial groups in part (a) of Theorem A are isomorphic to $C_3$, $C_7 \rtimes C_3$,
$\AAA_4$, $\SL{3}$, $(C_3 \times C_3) \rtimes \SL{3}$, $(C_2 \times C_2 \times C_2) \rtimes (C_7 \rtimes C_3)$, and $(C_5 \times C_5) \rtimes \SL{3}$.

\medskip

If $G$ is a finite group, then recall that the degrees of the irreducible real representations are   $\chi(1)$ if $\nu(\chi)=1$, $2\chi(1)$ if $\nu(\chi)=-1$, and $2\chi(1)$ for the representatives of the non-real irreducible characters of $G$. (See, for instance, page 108 of \cite{S}.)
Using this, we easily derive the following.

\begin{corB} \label{corB}
Let $G$ be a finite group. Suppose that the degrees of the inequivalent real irreducible representations of $G$ are all different.
\begin{enumerate}[\rm(a)]
\item
If $G$ is solvable and non-trivial, then $G$ is isomorphic to one of the following groups: $C_3$, $\AAA_4$, $C_7 \rtimes C_3$, or $(C_2 \times C_2 \times C_2) \rtimes (C_7\rtimes C_3)$.

\item
If $G$ is non-solvable, then $G$ is isomorphic to one of the following groups: $\AAA_8$, ${\rm M}_{11}$,
${\rm M}_{22}$, ${\rm M}_{23}$, ${\rm M}_{24}$, ${\rm McL}$, ${\rm Th}$, or ${\rm SL}_2(8).3$.
\end{enumerate}
Conversely, if $G$ is any of the groups listed in {\rm (a)} and {\rm (b)}, then the degrees of the inequivalent real irreducible representations of $G$ are all different.
\end{corB}


\section{Solvable groups}

For convenience, we say that a finite group $G$ satisfies {\bf {\em Hypothesis C}} if
any two characters $\alpha, \beta \in \irr G$ with the same Frobenius--Schur indicator and degree are complex-conjugate.  

\medskip

We  remark that if $G$ satisfies  Hypothesis C,  then every factor group of $G$  also satisfies Hypothesis C, and 
moreover for every positive integer $d$,
$\irr G$ contains at most four characters of degree $d$. 
We also have~$\Oh 2 G=G$, since $G$ has no  non-trivial real-valued  linear character.
(Recall that $\Oh p G$ is the smallest normal subgroup $N$
of the finite group $G$ such that $G/N$ is a $p$-group.)

\begin{lem}\label{l1}
  If $G >1$ is a nilpotent group that satisfies Hypothesis C, then $G$ is a cyclic group of order $3$. 
\end{lem}
\begin{proof}
Since $\Oh 2G=G$, we have that $G$ has odd order.
 Hence,  $G/G'$ has  odd order and  every non-principal character
$\lambda \in \irr{G/G'}$ is linear and non-real.  So $\irr{G/G'} = \{ 1, \lambda, \o\lambda\}$ and $G/G' = C_3$. 
(In this paper, $G'$ denotes the derived subgroup of $G$, that is, the smallest normal subgroup with abelian factor group.) In particular, if $\Phi(G)$ is the Frattini subgroup of $G$, then we have that  $G/\Phi(G)$ is cyclic, using that
since $G$ is nilpotent, then  $G' \sbs \Phi(G)$. Therefore $G$ itself must be cyclic, and hence  $G = C_3$.   
\end{proof}

In this paper, we use the notation for characters in \cite{Is}.
Thus, if $H \le G$ and $\tau \in \irr H$, then $\tau^G$ denotes the induced character from $H$ to $G$.
Also recall that non-trivial irreducible characters of groups of odd order are non-real (by a theorem of Burnside, see Problem 3.16 of \cite{Is}).

\begin{lem}\label{l2}
  Assume that $G$ satisfies Hypothesis C.  If $N \nor G$, $G/N$ has odd order and cyclic Sylow subgroups, and $\tau \in \irr N$ is non-real,
    then $\tau^G \in \irr G$. 
  \end{lem}
\begin{proof}
  Let  $G_{\tau}$ be the inertia subgroup of $\tau$ in $G$.
  (That is, $G_\tau$ is the set of elements $g \in G$ such that
  $\tau^g=\tau$, where $\tau^g$ is the character of $N$ defined by $\tau^g(n)=\tau(gng^{-1})$.)  The assumption that   the Sylow subgroups of $G/N$ are cyclic implies that $\tau$ has an extension
  $\gamma \in \irr{G_{\tau}}$ (by Corollary 11.22 of \cite{Is}.). If  $G_{\tau} >N$ then, as $G_{\tau}/N \leq G/N$ is solvable,
  there exists a linear non-principal $\lambda \in \irr{G_{\tau}/N}$. Hence, by the Clifford correspondence (Theorem 6.11 of \cite{Is}),  Gallagher's theorem (Corollary 6.17 of \cite{Is}), and \cite[Lemma 2.1]{IN},
 the induced characters $\gamma^G$ and $(\lambda \gamma)^G$ of~$G$ are irreducible of the same Frobenius--Schur indicator and degree. Thus (since~$G$ satisfies Hypothesis C)  they are complex-conjugate.
  This implies that $\tau$ and  its complex conjugate $\o \tau$ are conjugate by some element $g\in G$. Then $\tau^{g^2}=\tau$.  Since  $G/N$ has odd order,  $\tau$ is real, and this is a contradiction. 
\end{proof}

\begin{lem}\label{oddorder}
  If $G$ is a non-trivial group of odd order that satisfies Hypothesis C, then $G$ is either cyclic of order $3$ or a Frobenius group of order $21$. 
\end{lem}
\begin{proof}
  Let $N=G'$, $K=N'$ and $L=K'$. We have that $G/N=C_3$.
Let $1 \ne \lambda, \nu \in \irr{N/K}$. 
Since $G/N$ has prime order, then either $\lambda$ extends to $G$ or $G_\lambda=N$ and therefore $\lambda^G$ is irreducible  (by the Clifford correspondence).
Since linear characters of $G$ have $N$ in their kernel, we then have that $\lambda^G$ and $\nu^G \in \irr G$ should be irreducible. They also have the same Frobenius--Schur indicator 0, because groups of odd order do not have
non-trivial real-valued irreducible characters.
So $\lambda^G$ and $\nu^G$ are either equal or complex-conjugate. Thus $\lambda$ is $G$-conjugate to $\nu$ or to $\bar\nu$, and it follows that the action of $G/N$ on $\irr{N/K}-\{1_N\}$ has two orbits of size 3. (Notice that $\nu$ and $\bar\nu$ cannot be $G$-conjugate, since $G$ has odd order.)
Thus $N/K$ has order 7, and $G/K= C_7 \rtimes C_3$. Suppose that $K>L$.
Let $1 \ne \lambda, \nu \in \irr{K/L}$. By Lemma \ref{l2},
we have that $\lambda^G, \nu^G \in \irr G$.
Since these characters have the same Frobenius--Schur indicator and degree, we conclude that they are equal or complex-conjugate. Thus, again,  $\lambda$ is $G$-conjugate to $\nu$ or to $\bar\nu$.  Since $G_\lambda=K$ (using that $\lambda^G$ is irreducible and Problem 6.1 of \cite{Is}), we have that the $G$-orbit of $\lambda$ has size 21. Therefore $K/L$ has order 43. Thus $K/L$ is cyclic of order 43. But then $(G/L)/\cent{G/L}{K/L}$ is abelian, because it is isomorphic to a subgroup of the automorphism group of a cyclic group. Then $N/L \sbs \cent{G/L}{K/L}$, and therefore
 $K/L$ is in the center of $N/L$. This is impossible since there are characters of $K/L$ that induce irreducibly
to $G/L$. 
\end{proof}

We write $\irrp G = \{ \chi \in \irr G \mid \nu(\chi) = 1\}$, and denote by $A\Gamma(2^3)$ the affine group  $C_2^3 \rtimes (C_7 \rtimes C_3)$.

\begin{prop}\label{2closed}
  If $G$ satisfies Hypothesis C and has a normal Sylow $2$-subgroup, then it is isomorphic to a
  factor group of either $\SL 3 = Q_8 \rtimes C_3$ or $A\Gamma(2^3)$.
\end{prop}
\begin{proof}
  Let $N$ be the normal Sylow $2$-subgroup of $G$. By Lemma~\ref{oddorder}, $G/N \cong C_3$ or $G/N \cong C_7 \rtimes C_3$.
  We can  assume $N \neq 1$.

  Assume first $G/N \cong C_3$. Then $N = G'$, since $\Oh 2G = G$. If $\lambda \in \irr N$ is  linear of order $2$, then $\lambda^G \in \irrp G$,  
  so by Hypothesis C and Clifford's Theorem there is a unique $G$-orbit of such characters. Hence, $N/\Phi(N) \cong C_2 \times C_2$ is an irreducible $T$-module, where
  $T \cong C_3$ is a Sylow $3$-subgroup of $G$. This implies that $N/N'$ is  an indecomposable $T$-module, 
  so it is homocyclic. 
  For $\lambda  \in \irr{N/N'}$ with $o(\lambda)>2$, it follows from Lemma~\ref{l2} that $\lambda^G$ is an irreducible character of $G$
  and it is non-real (by~\cite[Lemma 2.1]{IN}), so there are at most two $T$-orbits of characters of order $> 2$ in $\irr{N/N'}$.
  Since $N/N'$ is homocyclic and $C_{2^n} \times C_{2^n}$ has $2^{2n}-2^2 \geq 12$ linear characters of order $>2$
  when~$n > 1$, this implies $N/N'\cong C_2\times C_2$ and $G/N' \cong \AAA_4$.
  If $N' \neq 1$, then $N$ is either dihedral, semidihedral or a generalized
  quaternion group (by~\cite[Satz III.11.9]{HuI}) and, observing that ${\rm Aut}(N)$ is not a $2$-group, we conclude that $N \cong Q_8$ and $G \cong \SL 3$. 
  
  \medskip
  Assume now that $G/N \cong C_7 \rtimes C_3$ and let $H = TR$ be a complement of $N$ in $G$, where $T \cong C_7$ and $R \cong C_3$ are
  Sylow subgroups of $G$.
  By considering a suitable factor group of $G$, we may assume first that $N$ is minimal normal in $G$. We observe that $T$ does not
  centralize $N$, as otherwise (recalling that $\Oh 2G = G$) $N$ is a faithful irreducible $R$-module, so $N \cong C_2 \times C_2$ and  $G$ would be a Frobenius group with kernel $N \times T$, 
  so $G$ would have $9$ irreducible characters  of degree $3$, a contradiction.
  By~\cite[Theorem 15.16]{Is} 
  (with~$G =H$, $V=N$, $N=T$, $H=R$, and $H_0 = 1$),
  we have $\dim N = \dim \cent N1 = 3 \dim \cent NR$, and thus
    there exists a non-trivial element 
    $x \in \cent NR$, so by minimality
  $N = {\gen x}^G= {\gen x}^T \cong C_2 \times C_2 \times C_2$ and $G \cong A\Gamma(2^3)$.
  (Alternatively, the faithful irreducible representations of~$G/N$ over~$\FF_2$ all have dimension~$3$ and are unique
  up to automorphisms of~$G/N$.)

  In order to conclude the proof, we show that $N$ is minimal normal in $G$. We first observe that $N/\Phi(N)$ is an irreducible $H$-module: otherwise,
  since it is a completely reducible $H$-module,
  by the previous paragraph there exist distinct normal subgroups  $K_1$ and $K_2$ of  $G$ such that $G/K_i \cong A\Gamma(2^3)$, for $i = 1,2$,
  so, as $A\Gamma(2^3)$ has two irreducible non-real characters of degree $7$, there would be at least four of them in $G$, against Hypothesis C.

  As $N/\Phi(N)$ is irreducible, then $N/N'$ is an indecomposable $H$-module and, using Lemma~\ref{l2} as in the second paragraph of this proof, Hypothesis C implies that $N' = \Phi(N)$.
  We now show that $K = N'$ is trivial. Working by contradiction, we may assume (by taking a suitable quotient of $G$) that
  $K$ is minimal normal in $G$, so $K = \zent N$.
   For $1 \neq \mu \in \irr K$, by~\cite[Lemma 2.2]{W} there is a unique subgroup $U_\mu$ with $K \leq U_{\mu} \leq N$ such that
  $\mu$ extends to $\lambda \in \irr{U_{\mu}}$ and $\lambda$ is fully ramified in $N/U_{\mu}$ (for a definition of fully ramified characters see~\cite[Problem 6.3]{Is}).
  As $\mu$ does not extend to $N$ and $|N/U_{\mu}|$ is a square
  (using \cite[Theorem 6.18(b)]{Is}),
   it follows $|U_{\mu}/K| = 2$.
  So, $\irr{N|\mu} = \{\theta_1, \theta_2\}$ has order $2$,  $\theta_i(1) = 2$ and $G_{\theta_i} = G_{\mu}$ for $i=1,2$.
 Note that $K$ is a faithful $H$-module, otherwise $H_{\mu}$ has order divisible by $7$ for some $1 \neq \mu \in \irr K$, and this gives more than $4$ irreducible characters of $G$ with the same degree. 
In particular, we see that $T\in \syl 7H$  acts fixed point freely on $K$. 
  An application of~\cite[Theorem 15.16]{Is} to the action of $H$ on the dual module $K^*$ (as
  $\cent {K^*}T = \cent KT = 1$) yields the existence of a
  non-principal irreducible character $\mu$ of $K$ such that $R \leq H_{\mu}$. So, $G_{\theta_i} = NR$ for $\theta_i \in \irr{N|\mu}$,
  $i=1,2$, and hence $G$ has at least six irreducible characters of degree $14$, a contradiction. 
Hence $K= 1$ and and $G \cong A\Gamma(2^3)$.
\end{proof}

In the following, we denote by $\Oh{2'}G$ the smallest normal subgroup of $G$ with odd index in $G$, and by $\oh 2G$ the largest normal subgroup of $G$ having order a power of $2$.

\begin{thm}\label{solvable}
  Let $G$ be a solvable group that satisfies Hypothesis C. Then $G$ is isomorphic to a factor group of either $A\Gamma(2^3)$ or 
  $(C_p\times C_p) \rtimes \SL 3$ for $p \in \{3, 5\}$.
\end{thm}
\begin{proof}
  Let $N = \Oh{2'}G$ and $K = \Oh 2N$. 
  If~$K=N$, then~$N$ has no quotients of order~$2$ or of odd order $>1$;
  since~$N$ is assumed to be solvable, this implies that~$N$ is trivial
  and we are done by Lemma~\ref{oddorder}; hence  we can assume~$K < N$.
By  Proposition~\ref{2closed}, we deduce that~$G/K$ is isomorphic to either~$A\Gamma(2^3)$, $\AAA_4$  or~$\SL 3$, and hence we can assume that~$K \neq 1$.

  \medskip
  {\bf (I.)}  Suppose first  $G/K \cong A\Gamma(2^3)$ and, working by contradiction,  that $K \neq 1$.
  By considering a suitable factor group of $G$, we can assume that $K$ is a minimal normal subgroup of $G$. Let $|K| = p^n$, for an  odd prime $p$.
We claim that $K$ is the Fitting subgroup~$\fit G$. Otherwise, $N \leq \fit G$, because $N/K$ is the only minimal normal subgroup of $G/K$. But then~$N$ would be nilpotent and
thus isomorphic
to a direct product of~$K$ and~$\oh 2N \cong N/K$,
and $G/\oh 2N$ would be a group of odd order $21|N|$ that satisfies Hypothesis C, contradicting Lemma~\ref{oddorder}.
Let  $M \in \syl 2N$ and  $H = \norm G M$. Then $H$ is a complement of $K$ in $G$,
  as $G = KH$ by the Frattini argument and $K \cap H = \cent KM = \zent N \cap K \nor G$, so $K \cap H = 1$ as $K$ is minimal normal in $G$. Let $T \cong C_7$ be a Sylow $7$-subgroup of $H$. Then $MT$ is a Frobenius group and hence $\cent KT \neq 1$ by~\cite[Theorem 15.16]{Is}.
  Thus, by~\cite[Theorem 6.32]{Is} there exists a non-principal $\mu \in \irr K$ such that $T \leq H_{\mu}$. Note that the inertia subgroup $N_{\mu}$ is a
  proper subgroup of $N$, since $\mu$ extends to $N_{\mu}$ by~\cite[Corollary 6.28]{Is} and $\mu$ has no estension $\nu \in \irr N$, as otherwise $K = N' \leq \ker{\nu_K} = \ker{\mu}$ against $\mu \neq 1_K$.
  Hence, $N_{\mu} = K$ because $N_{\mu}$ is
  $T$-invariant and $N/K$ is an irreducible $T$-module. Thus $\theta = \mu^N \in \irr N$. We show that
  $\theta$ is non-real. In fact, otherwise there exists an element $x \in M$ (as $M$ is a Sylow $2$-subgroup of $G$ and $KM = N$ contains every $2$-element of $G$) such that
  $\mu^x = \o{\mu} = \mu^{-1}$. So, considering the action of $H$ on the dual module $K^*$,  $\norm M{\gen{\mu}}$ is a non-trivial and $T$-invariant subgroup of $M$.
  Hence $\norm M{\gen{\mu}} = M$,  since $M$ is an irreducible $T$-module. But then $\cent M{\gen{\mu}}$ is a $T$-invariant subgroup
  of index $2$ in $M$, a contradiction. As $\theta \in \irr N$ is non-real, then $\theta^G \in \irr G$ by Lemma~\ref{l2}, which
  is a contradiction as $T \leq G_{\mu} \leq  I_G(\theta)$. Therefore, $K = 1$ if $G/K \cong A\Gamma(2^3)$.

  \medskip
{\bf (II.)}  We now assume $G/K \cong \AAA_4$ and, again, that $K$ is a minimal normal subgroup of $G$, $|K|$ odd. Let $M \in \syl 2G$.
As above, $K = \fit G$, as otherwise $N = M \times K$, with $M \nor G$, against Lemma~\ref{2closed}. 
Since $N = KM$ and $\cent KN \nor G$, by the minimality of $K$ we deduce that
$\cent KM = \cent KN = 1$. By the Frattini argument $G = KH$  where $H = \norm GM$ and $K \cap H = \cent KM = 1$, so $ H= MR \cong \AAA_4$ is a complement of $K$ in $G$, where $R\in \syl 3G$. 
  Then, considering the action of $H$ on the dual module $K^*$, $\cent{K^*}M = 1$ and  by~\cite[Theorem 15.16]{Is} there exists a non-principal and $R$-invariant $\mu \in K^*$ and, as $M_{\mu}$ is an
  $R$-invariant and proper  subgroup of $M$ (as $\cent{K^*}M = 1$), we deduce that
  $M_{\mu} = 1$ and hence $\theta = \mu^N \in \irr N$. As in the previous paragraph, one shows that $\theta$ is non-real, giving a
  contradiction by Lemma~\ref{l2}.

  \medskip
      {\bf (III.)}  Finally, we assume $G/K \cong \SL 3$ and we prove, by induction on $|G|$, that either
  $G \cong G_1 = (C_3 \times  C_3)\rtimes {\rm SL}_2(3)$ or  $G \cong G_2 =  (C_5 \times  C_5)\rtimes {\rm SL}_2(3)$    . 

      To start with, we prove that $K$ is nilpotent.
Working by contradiction, we  assume that $F = \fit K < K$ and,  
by considering a suitable factor group of $G$, that $F$ is minimal normal in $G$.
By induction, $G/F \cong G_i$ for some  $i \in \{ 1,2\}$.
Write $|F| = q^m$, where $q$ is a prime number,  and observe that $q \neq p$, where $p = 3$ if $i=1$, and $p = 5$ if $i=2$.
For  $M \in \syl pK$, $\cent FM = \cent FK = 1$ by the minimality of $F$, and 
 the Frattini argument implies that  $H = \norm GM$ is a complement of $F$ in $G$. Let $L = N \cap H$; then $L \cong (C_p\times C_p) \rtimes Q_8$ has index $3$ in $H$. 
In order to get a contradiction, we will show that then $L $ has at least $14$ regular orbits on the dual module $V = F^*$.
In fact, the regular $L$-orbits are permuted by  the action of $\overline{H} = H/L \cong C_3$, and the ones that are not fixed
by $\overline{H}$ give, in groups of three, regular orbits of $H$ on $V$.
If, instead, $\Delta$ is a regular $L$-orbit in $V$ such that $\Delta$ is $\overline{H}$-invariant, then $\Delta$ is an $H$-orbit
of size $|L|$. Thus, for $\lambda\in \Delta$, $|H_{\lambda}|= 3$  and by Clifford Correspondence $\irr{G|\lambda}$ contains three
irreducible characters of degree $|L|$. 
Hence, if $L$ has at least $14 =  4 \times 3 +2$ regular orbits on $V$, then $G$ has more than four 
irreducible characters either of degree  $|H|$ or of degree $|L|$, against Hypothesis C.

Assume first $p =3$: so $L$ has four subgroups of order $3$ and nine subgroups of order $2$. These are the only subgroups of prime order
of $L$ and hence every element in $F$ that is not centralized by any of them is in a regular $L$-orbit.
If $X\leq L$ and  $|X| = 3$, then $|\cent VX| \leq q^{m/2}$ because two distinct subgroups $X_1, X_2$ of order $3$ of $L$ generate $K \cap H \cong C_3 \times C_3$,
so $\cent V{X_1} \cap \cent V{X_2} = \cent V{K \cap H} = \cent FK = 1$.
If $X\leq L$ and  $|X| = 2$, then $|\cent VX| = q^{m/2}$ by~\cite[Theorem 15.16]{Is}.
Moreover, in characteristic $q\neq 3$, all faithful  absolutely irreducible representations of
$H \cong G_1$ have degree $8$, as  one can check by looking at its  (ordinary and) Brauer character tables, so $m$ is a multiple of $8$.
Now, the function
$$g(q, m) = (q^m -1) - (9+4)\cdot(q^{m/2} -1) - 14|L|$$
is $> 0$ for all primes $q$ and $m \geq 8$, except for the pair $(q, m) = (2,8)$.
Using~\cite{GAP}, one checks that $H$ has a unique faithful (absolutely) irreducible module $F_0$ of dimension $8$ over $\FF_2$ and
that the semidirect product $G_0 = F_0 H$ has $6$ irreducible characters of degree $72$, so does not satisfy Hypothesis C. Hence, $p\neq 3$.  

Finally, if $p =5$ then $L$ has $6$ subgroups of order $5$ and $25$ subgroups of order $2$.
As in the previous paragraph, one checks that for $X \leq L$, $|X| \in \{2,5\}$, $|\cent FX| \leq |F|^{\frac{1}{2}}$.
Moreover, from the Brauer
character tables of $H \cong G_2$  it turns out that,  for every characteristic  $q \neq 5$,  $H$ has
just one faithful absolutely irreducible representation, which has dimension $24$. So $m$ is a multiple of $24$,  and one easily checks that
$$h(q, m) = (q^m -1) - (6+25)\cdot(q^{m/2} -1) - 14|L|$$
  is positive for all primes $q$ and integers $m\geq 24$.

   \smallskip
Hence, we have proved that $K$ is nilpotent. It follows that $|K|$ is odd and that $K = \fit G$, as otherwise $\oh 2G \neq 1$ and  factor group $G/\oh 2G$ contradicts the inductive hypothesis. 
We now prove that $K$ is abelian. Working again by contradiction, we assume that $X = K' \neq 1$.
By induction, $X$  is a minimal normal subgroup of $G$ and  $G/X$ is isomorphic to either $G_1$ or $G_2$.
Hence, $K/X$ is a  minimal normal subgroup of $G/X$ and it follows that $X = \zent K = \frat K$. 
  Since $|K/X| \in \{ 3^2, 5^2\}$, we deduce that  $K$ is an extraspecial group of order $p^3$, $p \in \{3, 5\}$.
Moreover, $G/\cent GX$ is an abelian group of order dividing $|G/G'| = 3$, so  $X \leq \zent G$. Thus, the $p-1$ irreducible
  characters of degree $p$ of $K$ extend to $G$ by \cite[Corollary 11.31]{Is}, as $G/K$ has either cyclic or quaternion Sylow subgroups,
  and hence by Gallagher's theorem (\cite[Corollary 6.17]{Is}) $G$ has $3(p-1) \geq 6$ irreducible characters of degree $p$, a contradiction. Thus, $K$ is abelian. 

 Let  $Q \cong Q_8$ be a Sylow $2$-subgroup of $G$ and $Z = \zent Q$.
 As  $K$ is an abelian of odd order, then $K = \cent KZ \times [K,Z]$. Since $KZ \nor G$, $[K, Z] = [K, KZ] \nor G$ and then
 $L =  [K, Z]Z \nor G$, since $L/[K, Z]$ is the Sylow $2$-subgroup of the abelian normal subgroup $KZ/[K, Z]$ of $G/[K,  Z]$.
 So $\overline{G} = G/L$ satisfies Hypothesis C and $\overline{G}/\overline{K} \cong \AAA_4$,
 with $\overline{K} = \Oh 2{\Oh{2'}{\overline{G}}}$, and hence by case (II) $\overline{K} = 1$, which gives  $\cent KZ = 1$. 
 So,
 the involution $z \in Z$ acts by inverting all elements of $K$ and,
 in particular, $\cent KQ = 1$.
 
Thus, by the Frattini argument $H = \norm GQ$ is a complement of  $K$ in $G$. Note that $H = QR \cong \SL 3$, where $|R| = 3$.
  If $\mu \in \irr K$ is such that $H_{\mu} = 1$ (so $G_{\mu} = K$), then $\chi = \mu^G \in \irr G$ and $\chi$ is a real character
  as $\o{\mu} = \mu^{-1} = \mu^z$. By~\cite[Lemma 10.8]{Is}, the Schur index $m_{\mathbb{R}}(\chi)$ divides $\mu(1)= 1$ and hence
  $\chi \in \irrp G$ by~\cite[Theorem 10.17]{Is}. 
 Therefore, by Hypothesis C there is at most one regular orbit of $G/K$ on $\Irr(K)$.

\smallskip
Next, we show  that $K$ is a group of prime power order.
Otherwise, the  inductive hypothesis  applied to 
$G/{\bf O}_{p}(K)$ for each prime divisor $p$ of $|K|$ yields  $G \cong (A \times B) \rtimes \SL3$, with
$A \cong C_3 \times C_3$ and $B \cong C_5 \times C_5$.
Moreover, $G/K$ acts regularly on $\Irr(B) \smallsetminus \{1_B\}$ and hence, 
for any $\alpha \in \Irr(A)$ and any $\beta \in \Irr(B)$, $\beta \neq 1_B$,  $\alpha \times \beta$ has stabilizer $A \times B = K$ in $G$. Thus $G/K$ has $9$ regular orbits on $\Irr(K)$, a contradiction by the previous paragraph.

Thus, $K$ is an abelian  $p$-group for some prime $p \neq 2$. Let $|K| = p^n$, for some positive integer $n$,
and let $V = K^*$ be the dual $H$-module. Note that, as $K = \fit G$, both $K$ and $V$ are faithful $H$-modules. 
Next, we show that 
$$(p,n) \in \{(3,2), (5,2), (7, 2) \}.$$
  Let $\syl 3H = \{R_i| 1 \leq i \leq 4 \}$ and observe that $\cent V{R_i} \cap \cent V{R_j} = 1$ for $i\neq j$, because
  $\gen{R_i, R_j} = H$ and $\cent V{R_i} \cap \cent V{R_j} = \cent V{\langle R_i, R_j\rangle} =  \cent KH \leq \cent KQ= 1$.
  Thus,  $|\cent V{R_i}| \leq p^{[n/2]}$.
  Define
  $f(p, n) = p^n - 4p^{[n/2]} -21$ and observe that, writing  $V^{\#} = V \smallsetminus \{ 1_K \}$, 
\begin{equation}\label{eq1}
  |V^{\#}| - |\bigcup_{i=1}^4 \cent{V^{\#} }{R_i}| = (p^n -1) - 4(p^a -1)
\end{equation}
  is larger than $|H| = 24$ if  $f(n,p) >0$.
  So, if  $f(p,n) > 0$ then  $H$ has at least two regular orbits on $V$ and hence, as observed above, $G$ does not satisfy Hypothesis C.
  We also observe that $n$ is even, as every indecomposable summand of $V$ as a  $Q$-module is a homocyclic group of even rank because $\cent VQ = \cent KQ = 1$. In particular,  $n \geq 2$. 

  One can check that $f(p,n)>0$ for all pairs $(p,n)$ with $p > 7$ and $n \geq 2$, or such that $n \geq 4$ (and $p\neq 2$), thus 
  proving the claim.
So, considering that $K$ is non-cyclic because  $H$ acts faithfully on $K$,  we deduce that $K$ is an elementary abelian group, with $|K| \in \{ 3^2, 5^2, 7^2 \}$.
  
If $|K|= 7^2$, then $G$ is either $\SmallGroup(1176, 214)$ or  $\SmallGroup(1176, 215)$ (in the {\GAP} library of small groups~\cite{SmallGroup}). In fact, $H$ has three isomorphism classes of faithful (irreducible) modules $K$ of dimension $2$ over $GF(7)$ and  the corresponding semidirect products $G = K \rtimes H$ belong to the two stated isomorphism types. In the first case,  $H$ has three orbits of size $8$ on $K^*$, so by Clifford theory $G$ has nine irreducible characters of degree $8$.
In the second case, $H$ acts semiregularly on $K^*$, so it has two regular orbits on $K^*$.
Hence,  both  these groups do not satisfy Hypothesis C.

Finally, $H$ has exactly one isomorphisme type of faithful (irreducible) modules of dimension $2$ over both $GF(3)$ and $GF(5)$. The corresponding semidirect products $G = K \rtimes H$ are, respectively, isomorphic to $G_1$ and $G_2$, concluding the proof. 

We also observe that  both $G_1 = \SmallGroup(216, 153)$ 
and $G_2 = \SmallGroup(600, 150)$ (in~\cite{SmallGroup}) satisfy Hypothesis C.
\end{proof}

\section{Almost Simple groups}

\begin{thm}\label{as}
  If $G$ is an almost simple group and $G$ satisfies Hypothesis C, then $G$ is one of the groups in the list $\mathcal{L}$ defined in \eqref{list1}.
\end{thm}

\begin{proof}
(a) Let $S$ denote the socle of $G$, so that $S \lhd G \leq \Aut(S)$. First we consider the case $S = \AAA_n$ with $n \geq 5$. Since the cases 
$5 \leq n \leq 12$ can be checked directly using \cite{atlas}, we will assume that $n \geq 13$, and hence $G=S$ since $\SSS_n$ does not satisfy 
Hypothesis C. To rule out such groups $G$, 
it suffices to find a partition $\lambda \vdash n$ which is self-associate, such that the Young diagram of $\lambda$ has $k$ nodes 
on the main diagonal and $k \equiv n \pmod{4}$. Indeed, in this case, the character $\chi \in \Irr(\SSS_n)$ labeled by $\lambda$ splits upon 
restriction to $\AAA_n$ into the sum $\chi^++\chi^-$ of two real-valued irreducible characters of $\AAA_n$, which are $\SSS_n$-conjugate and hence 
of the same degree and same Frobenius--Schur indicator, see \cite[Theorem 2.5.13]{JK}. Now, if $n \equiv 1 \pmod{4}$, we take 
$\lambda = ((n+1)/2,1^{(n-1)/2})$, the symmetric hook partition of $n$. If $n \equiv 2 \pmod{4}$, take 
$\lambda = (n/2,2,1^{n/2-2})$, the symmetric hook partition of $n-1$ augmented by a node in the second row. If $n \equiv 3 \pmod{4}$, take 
$\lambda = ((n-3)/2,3,3,1^{(n-9)/2})$, the symmetric hook partition of $n-4$ augmented by a $2 \times 2$-square in the second and third rows.
If $n \equiv 0 \pmod{4}$, take 
$\lambda = ((n-8)/2,4,4,4,1^{(n-16)/2})$, the symmetric hook partition of $n-9$ augmented by a $3 \times 3$-square in the second, third, and fourth rows.

\smallskip
(b) Checking the almost simple groups with sporadic socle $S$ using \cite{atlas}, we may now assume that $S$ is a simple group of Lie type. Next we observe
that $[\QQ(\chi):\QQ] \leq 2$ for any $\chi \in \Irr(G)$. Indeed, the $\Gamma$-orbit of $\chi$ has length
$|\Gamma|$, where $\Gamma:= \Gal(\QQ(\chi)/\QQ)$. Clearly, any Galois conjugate of $\chi$ has the same degree and Frobenius--Schur indicator as  $\chi$. Hence Hypothesis
C implies that $2 \geq |\Gamma| = [\QQ(\chi):\QQ]$.

Thus $G$ is a {\sl quadratic-rational group}, in the sense of \cite{Tr}. Therefore, all the possibilities for $S$ are listed in \cite[Theorem 1.2]{Tr}; furthermore,
$G/S$ is a solvable group with no normal subgroup of index $2$ or $5$ (by Hypothesis C applied to $G/S$). All
such groups $G$ can be checked using {\GAP}. 
\end{proof}

\section{Modules and the Almost Simple Groups}

\begin{thm}\label{modules}
Suppose that $\tilde{G}$ is a finite group having
a unique minimal normal subgroup $N$ such that $N$ is an
irreducible faithful module for $G = \tilde{G}/N$
and $G$ is one of the groups in the list $\mathcal{L}$ defined in \eqref{list1}.
Then $\tilde{G}$ does not satisfy Hypothesis C.
\end{thm}

\begin{proof}
The proof is computational.
In the computations, we used the character tables and tables of marks
from the {\GAP} packages~\cite{CTblLib} and~\cite{TomLib}, respectively.
The matrix representations were taken from the
Atlas of Group Representations~\cite{AGR}
or derived from representations available there.
Cohomology groups and group extensions were computed using implementations
in~\cite{Cohomolo,GAP,OSCAR}.
Permutation representations of small degrees of the group extensions
were obtained using ideas from~\cite{Hulpke},
and~\cite{Magma} was used to compute the character tables.

By assumption, $N$ is an irreducible faithful $G$-module
of dimension $n$ over a prime field~$\FF = \FF_p$ with~$p$ elements.
Viewing $\Irr(N)$ as a vector space over~$\FF$, we can identify it with the dual module $M =N^*$.
By Clifford theory,
each regular orbit of $G$ on $M$ yields an irreducible character
of degree $|G|$ of $\tilde{G}$, see~\cite[Theorem~6.11]{Is}.
Thus we are done if we show that $G$ has at least five regular orbits
on $M$.
We use the techniques from~\cite{FOS} and~\cite{FMOW} to derive
lower bounds for the number $R(G, M)$ of regular orbits, as follows.

By construction, $M$ is an irreducible faithful $G$-module of dimension $n$
over~$\FF$. 
Let $P$ be a set that contains one generator for each cyclic subgroup
of prime order in $G$.
Then
\[
   M = \bigl\{ v \in M; |\Stab_G(v)| = 1 \bigr\} \cup
       \left( \bigcup_{g \in P} \Fix_M(g) \right)
\]
implies
\[
   |M| \leq |G| \cdot R(G, M) + \sum_{g \in P} |\Fix_M(g)|.
\]

Let $T$ denote the socle of $G$.
For a nonidentity element $g \in G$,
define $r(g)$ to be the minimal number of $T$-conjugates of $g$
that generate $\langle T, g \rangle$,
and set 
$$r(G) = \max\{ r(g); g \in G, |g| \ne 1 \}.$$
Then~\cite[Lemma~3.4]{FOS} yields
$$\dim_{\FF}(\Fix_M(g)) \leq \dim_{\FF}(M) (1 - 1/r(g)).$$
The values $r(g)$ for the groups listed in the theorem
are given by~\cite[Theorem~1.3]{FMOW} for sporadic simple groups,
and have been computed directly for the other groups in question;
they can be read off from Table~\ref{table:rg}, in the following sense:
We always have $r(g) \geq 3$ if $g$ has order $2$,
and $r(g) \geq 2$ otherwise;
Table~\ref{table:rg} lists all cases where the inequality is strict.

\begin{table}
\begin{center}
\[
\begin{array}{|l|l|l|l|} \hline
   G                & r(g) = 3 & r(g) = 4 & r(g) = 5 \\ \hline
   \AAA_8           &          & 2A,2B,3A &          \\
   {\rm M}_{22}     &          & 2B       &          \\
   {\rm SU}_3(3)    & 3A       & 2A       &          \\
   {\rm McL}        & 3A       &          &          \\
   {\rm O}_8^+(2).3 & 3D,3E,3F & 3ABC     & 2A,2BCD  \\ \hline
\end{array}
\]
\caption{Exceptional values of $r(g)$.}
\label{table:rg}
\end{center}
\end{table}

It follows that
\[
   |G| \cdot R(G, M) \geq p^n - \sum_{g \in P} p^{n(1 - 1/r(g))}
                     \geq p^{n(1 - 1/r(G))}
                          \left(p^{n/r(G)} - |P|\right).
\]
Thus $R(G, M) \geq 5$ holds except for only finitely many pairs $(p, n)$.

We define $D(G)$ to be the smallest integer $n$ such that
$$5 |G| \geq 2^{n(1 - 1/r(G))} \left(2^{n/r(G)} - |P|\right)$$ 
holds.
Then $R(G, M) \geq 5$ holds for any $G$-module $M$ of dimension
larger than $D(G)$ over a finite prime field.

The Brauer characters of those irreducible $G$-modules $M$
for which this crude estimate
does not suffice are listed in~\cite{ABC} and (if the characteristic $p$
does not divide $|G|$) in~\cite{atlas}.
For each $p'$-element $g$, we can compute the exact value of
$\dim_{\FF}(\Fix_M(g))$ from the Brauer character of $M$,
which yields better estimates for $R(G, M)$ and rules out many cases.
The modules for which this does not establish at least five regular orbits
are listed in Table~\ref{table:modules_direct}.
Direct computations with the table of marks of $G$ and with matrices that
describe the $G$-action on $M$ yield the orbit lengths.
Finally, if this is not sufficient, we compute the possible extensions
$\tilde{G}$ of $G$ by $N \cong M^*$,
and verify directly that these groups $\tilde{G}$ do not satisfy Hypothesis~C.

\smallskip
For example, let $G$ be the Thompson group Th.
We have 
$$r(G) = 3\mbox{ and }|P| = 675\,176\,077\,846\,831,$$
thus $p^{n/r(G)} - |P| > 0$ if $n \geq 148$ holds,
and $2^{2n/3} (2^{n/3} - |P|) > 5 |G|$ in this case; hence $D(G) = 148$.
Since the minimal degree of a nontrivial irreducible representation of $G$
in any characteristic is $248$ (see~\cite{Jansen}),
we have shown that the claim holds for the group Th.

\smallskip
For the group $G = \mathrm{O}_8^+(2).3$, we compute $D(G) = 113$,
thus the five modules listed in Table~\ref{table:modules_direct} must be
considered.
\begin{itemize}
\item
  The $24$-dimensional representation over $\FF_2$
  is induced from $\mathrm{O}_8^+(2)$,
  thus we consider the restriction of the possible group extensions
  (there are two, a split and a non-split one)
  to the index $3$ subgroup $2^{24}.\mathrm{O}_8^+(2)$.
  The $24$-dimensional module splits as a direct sum of three $8$-dimensional
  ones, thus there are three factor groups $2^8.\mathrm{O}_8^+(2)$,
  which are permuted cyclically by the automorphism of order three.
  The groups $2^8.\mathrm{O}_8^+(2)$ (there are two, a split and a non-split one)
  have the character degree $2835$ with multiplicity $6$,
  which yields the degree $3 \cdot 2835$ with multiplicity at least $6$
  for $2^{24}.\mathrm{O}_8^+(2).3$.
\item
  The restriction of the $26$-dimensional representation over $\FF_2$
  to $\mathrm{O}_8^+(2)$ has a unique orbit of length $1575$
  on the module.
  Hence also $G$ has this orbit.
  We compute that there is only one possible extension of $G$ with the module,
  which is split, thus it is enough to establish a big enough
  degree multiplicity among the ordinary irreducible characters of the point
  stabilizer, which is an involution centralizer of order $331776$.
  This subgroup has the irreducible degree $16$ with multiplicity $6$,
  which yields multiplicity at least $6$ for the degree $1575 \cdot 16$ in
  $2^{26}.\mathrm{O}_8^+(2).3$.
\item
  The restriction of the $52$-dimensional representation over $\FF_2$
  to $\mathrm{O}_8^+(2)$ is twice the
  $26$-dimensional representation which we dealt with above.
  This restriction has $25493462$ regular orbits, which is enough
  for $\mathrm{O}_8^+(2).3$.
\item
  The restrictions of the $G$-representations of the degrees $28$ and $48$
  over $\FF_3$ to $\mathrm{O}_8^+(2)$
  have $127877$ and $457929531422105$ regular orbits,
  which yields huge degree multiplicities also for $\mathrm{O}_8^+(2).3$.
\end{itemize}
In this case, we chose to work with restrictions to $\mathrm{O}_8^+(2)$ since
the table of marks of this simple subgroup is available,
but the table of marks for $\mathrm{O}_8^+(2).3$ is not.

\smallskip
For the other groups $G$ in question,
Table~\ref{table:modules_direct} describes the arguments that exclude the
modules, as follows.

\begin{table}
\begin{center}
\scriptsize
\[
\begin{array}[t]{|l|r|r|r|r|r|r|} \hline
G                & D(G) &  p &            M &   R_1 (R_2) & H^2 & \textrm{deg.} \\
\hline
\AAA_8           &   45 &  2 &  \texttt{4a} &             & 1 & +105^3 \\
                 &      &    &  \texttt{4b} &             & 1 & +105^3 \\
                 &      &    &  \texttt{6a} &             & 0 & 35^5 \\
                 &      &    & \texttt{14a} &             & 2 & 210^5 \\
                 &      &    & \texttt{20a} &    37       &   &   \\
                 &      &    & \texttt{20b} &    37       &   &   \\
                 &      &  3 &  \texttt{7a} &             & 1 & 28^5 \\
                 &      &    & \texttt{13a} &    23       &   &   \\
                 &      &  5 &  \texttt{7a} &             & 0 &  28^9 \\
                 &      &  7 &  \texttt{7a} &        (15) &   &   \\
                 &      & 11 &  \texttt{7a} &   240       &   &   \\
                 &      & 13 &  \texttt{7a} &  1122       &   &   \\
\hline
{\rm SL}_3(2)    &   18 &  2 &  \texttt{3a} &             & 1 & +21^2 \\
                 &      &    &  \texttt{3b} &             & 1 & +21^2 \\
                 &      &    &  \texttt{8a} &             & 0 & 24^7 \\
                 &      &  3 &  \texttt{6a} &             & 0 & 21^{24} \\
                 &      &    & \texttt{3ab} &             & 0 & 42^{16} \\
                 &      &  7 &  \texttt{3a} &             & 1 & 24^{14} \\
\hline
{\rm M}_{11}     &   30 &  2 & \texttt{10a} &             & 1 & 55^5 \\
                 &      &  3 &  \texttt{5a} &             & 0 & +110^2 \\
                 &      &    &  \texttt{5b} &             & 0 & 220^5 \\
                 &      &    & \texttt{10a} &             & 0 & 55^9 \\
                 &      &    & \texttt{10b} &     5       &   &   \\
                 &      &    & \texttt{10c} &     5       &   &   \\
\hline
{\rm M}_{22}     &   64 &  2 & \texttt{10a} &             &   & +1155^2 \\
                 &      &    & \texttt{10b} &             &   & +1155^2 \\
                 &      &  3 & \texttt{21a} & 20241       &   &   \\
\hline
{\rm M}_{23}     &   59 &  2 & \texttt{11a} &             & 1 & +253^2 \\
                 &      &    & \texttt{11b} &             & 0 & +253^3 \\
\hline
\end{array} \  \begin{array}[t]{|l|r|r|r|r|r|r|} \hline
G                & D(G) &  p &            M &   R_1 (R_2) & H^2 & \textrm{deg.} \\
\hline
{\rm M}_{24}     &   68 &  2 & \texttt{11a} &             & 0 & +239085^4 \\
                 &      &    & \texttt{11b} &             & 1 & +239085^2 \\
                 &      &    & \texttt{44a} & 69255       &   &   \\
                 &      &    & \texttt{44b} & 69173       &   &   \\
                 &      &  3 & \texttt{22a} &    34       &   &   \\
\hline
{\rm SU}_3(3)    &   38 &  2 &  \texttt{6a} &             & 1 & 189^6, +378^2 \\
                 &      &    & \texttt{14a} &             & 0 & 252^{13} \\
                 &      &  3 & \texttt{3ab} &             & 0 & 224^9 \\
                 &      &    & \texttt{6ab} &    40       &   &   \\
                 &      &    &  \texttt{7a} &             & 2 & 672^9, 756^8 \\
                 &      &  5 &  \texttt{6a} &             & 0 & 756^8 \\
                 &      &    &  \texttt{7b} &             & 0 & 108^8 \\
                 &      &    &  \texttt{7c} &             & 0 & 108^8 \\
                 &      &  7 &  \texttt{6a} &    16       &   &   \\
\hline
{\rm McL}        &   77 &  2 & \texttt{22a} &             & 1 & 311850^6 \\
                 &      &  3 & \texttt{21a} &        (10) &   &   \\
\hline
{\rm Th}         &  148 & & \textrm{(none)} &             &   &   \\
\hline
{\rm SL}_2(8).3  &   24 &  2 &  \texttt{6a} &             & 1 & 63^6 \\
                 &      &    & \texttt{12a} &             & 0 & 63^6 \\
                 &      &    &  \texttt{8a} &             & 0 & 36^6 \\
                 &      &    & \texttt{8bc} &    34       &   &   \\
                 &      &  3 &  \texttt{7a} &             & 2 & 56^9 \\
\hline
{\rm O}_8^+(2).3 &  113 &  2 & \texttt{24a} &             &   &   \\
                 &      &    & \texttt{26a} &             &   &   \\
                 &      &    &\texttt{26bc} &             &   &   \\
                 &      &  3 & \texttt{28a} &             &   &   \\
                 &      &  3 & \texttt{48a} &             &   &   \\
\hline
\end{array}
\]
\caption{Modules to be considered.}
\label{table:modules_direct}
\end{center}
\end{table}

\begin{itemize}
\item
  The column with header $p$ lists the characteristics for which at least
  one irreducible $\FF G$-module $M$ exists such that the character-theoretic
  lower bound for the number of regular orbits is less than $5$.
\item
  The column with header $M$ lists the irreducible modules to be considered.
  Names consisting of an integer $n$ followed by one letter
  \texttt{a}, \texttt{b}, etc. denote the first, second, etc.
  absolutely irreducible module of dimension $n$,
  names involving several letters denote modules that decompose into
  a sum of the corresponding absolutely irreducible modules of dimension $n$
  over an extension field; for example \texttt{3ab} means a module
  that is irreducible over the prime field with $p$ elements, of dimension $6$, 
  and decomposes into two absolutely irreducible modules over the field
  with $p^2$ elements.
\item
  The column with header $R_1 (R_2)$ contains the number of regular orbits,
  provided this number is at least $5$,
  or in brackets the number of orbits of length $|G|/2$,
  provided this number is at least $9$.
  Note that each orbit of length $|G|/2$ corresponds to an inertia subgroup
  which has only the character degrees $1$ and $2$,
  hence $9$ orbits yield one character degree with multiplicity at least $5$.

  The values in this column have been computed from the known table of marks
  of $G$, together with matrix generators describing $M$.
\item
  If the column $R_1 (R_2)$ does not exclude the module in question then
  the column with header $H^2$ contains the dimension of the corresponding
  second cohomology group that describes the possible group extensions,
  and the column with header deg. contains a description of
  character degrees that exclude the module.
  In this column deg., the entry $d^n$ means that the character degree $d$ occurs
  with multiplicity $n$,
  and the entry $+d^n$ means that degree $d$ occurs
  with multiplicity $n$ for characters with indicator $+$.
\end{itemize}
\end{proof}

\section{Non-Solvable groups}

\begin{lem}\label{trick1}
  Let $S$ be a non-abelian simple group and $A = \aut{S}$. Then,  there exists a non-principal
  character $\phi \in \irrp S$ having  an extension  $\hat{\phi} \in \irrp{A}$.
\end{lem}
\begin{proof}
  If $S$ is a simple group of Lie type, not isomorphic to $^2F_4(2)'$,   we let $\phi$ be the Steinberg character of $S$. It is  known (\cite[Theorem~B]{F}) that $\phi$ has an extension $\hat{\phi} \in \irrp{\aut S}$.
  
  For $S =  \AAA_n$,  $n \geq 5$, $n \neq 6$ ($\AAA_6 \cong {\rm PSL}_2(9)$) we consider the (unique) character
  $\phi \in \irr{\AAA_n}$ such that $\phi(1) = n-1$; then it is well known that $\phi$ has an extension $\hat{\phi} \in \irrp{\aut S}$.

  If  $S$ is  a  sporadic simple group or the Tits group, we choose the unique $\phi \in \irr S$ of degree $d$, where the pairs $(S,d)$ are as follows (see~\cite{atlas}):
  
$({\rm M}_{11},11 )$,  $({\rm M}_{12},45 )$,  $({\rm M}_{22},55 )$,  $({\rm M}_{23},22 )$,  $({\rm M}_{24},23 )$,   $({\rm J}_{1},209 )$,   $({\rm J}_{2},36)$,   
$({\rm J}_{3},324)$,
  $({\rm J}_{4},889111)$,  $({\rm Co}_{1}, 276)$,    $({\rm Co}_{2}, 23)$,    $({\rm Co}_{3},23 )$, $({\rm Fi}_{22},78 )$,
  $({\rm Fi}_{23},782 )$, $({\rm Fi}'_{24},8671 )$,
  $({\rm HS}, 22)$, $({\rm McL}, 22)$, $({\rm He}, 680)$,  $({\rm Ru}, 406)$,  $({\rm Suz},143 )$, $( {\rm O'N}, 10944)$,
  $({\rm HN}, 760)$,  $({\rm Ly},45694 )$,  $({\rm Th},248 )$, $({\rm B},4371 )$, $({\rm M}, 196883)$,
$({}^2{\rm F}_4(2)', 78)$.    
\end{proof}

For a finite set $X$ and an integer $1 \leq k \leq |X|$, we denote by  
$$X^{[k]} = \{ Y \subseteq X \mid |Y| = k \}$$
the set consisting of the  subsets of cardinality $k$  of $X$.
Given an action of a group $G$ on a set $X$ and a subset $Y \in X^{[k]}$, we denote by $G_Y = \{ g \in G \mid Y^g= Y \}$ 
the stabilizer of $Y$ in the natural action of $G$ on $X^{[k]}$.
  
The following result will be useful in dealing with some of the cases arising in the proof of Theorem~A. 

\begin{lem}\label{L1}
  Let $G$ be a finite group and let $M$ be a non-solvable minimal normal subgroup of $G$ such that $\cent GM = 1$. 
Writing $M = S_1 \times S_2 \times \cdots \times S_n$,  with $S_i$  isomorphic non-abelian simple groups,  
we consider the action of $G$  on $X =  \{1, 2, \ldots, n\}$ defined, for $g \in G$ and
$i \in X$, by
$$S_{i^g} = (S_i)^g \; .$$
If $G$ satisfies Hypothesis C, then

\begin{enumerate}
\item[\rm(a)] For every non-empty $Y \subseteq X$,   either
  $G_Y$ is a perfect group or $|G_Y: G_Y '| = 3$.   

\item[\rm(b)] 
 For every $1 \leq k \leq n$,   the orbits of $G$  on $X^{[k]}$ have distinct sizes.   
\end{enumerate}
\end{lem}
\begin{proof}
  Let $X = \{1, 2, \ldots, n\}$,  $S = S_1$ and $A = \aut{S}$. 
  Since $\cent GM = 1$, $G$ is isomorphic to a subgroup of $\aut{M}$. Hence, we can identify $G$ with a subgroup of
   the wreath product 
   $W = A \wr \Sym(X)$ (\cite[3.3.20]{R}) and  
   the action of $W$ on $X$, defined  by  $S_{i^w} = (S_i)^w$ for $w \in W$ and $i \in X$, 
extends  the action of $G$ on $X$.
   As $M$ is a minimal normal subgroup of $G$, $G$ acts transitively on $X$ and hence 
    there exist elements  $x_1 = 1, x_2, \ldots, x_n \in G$  such that $S_i = S^{x_i}$.
   So, the base group of $W$ is  $B = A_1 \times A_2 \times \cdots \times A_n$, where $A_i = A^{x_i}$.
  
  By Lemma~\ref{trick1} there exists a  character $\phi \in \irrp{S}$ having an extension $\hat{\phi} \in \irrp{A}$.
  So, $\hat{\phi}_i = \hat{\phi}^{x_i} \in \irrp{A_i}$ extends  $\phi_i = \phi^{x_i} \in \irr{S_i}$, for every $i \in X$.

  Let  $Y$ be a non-empty subset of $X$. We consider  the character
  $$\theta =  \prod_{i \in X}\alpha_i\in \irr M \text{, \  where  } \alpha_i =
 \begin{cases}
   \phi_i & \text{ if } i \in Y \\
   1_{S_i} & \text{ if } i \in X \smallsetminus Y \, .
 \end{cases}
$$ 

Hence, the character 
 $$\psi  = \prod_{i \in X}\beta_i  \in \irr B \text{, \  where  } \beta_i =
 \begin{cases}
   \hat{\phi}_i & \text{ if } i \in Y \\
   1_{A_i} & \text{ if } i \in X \smallsetminus Y 
 \end{cases}
$$ 
is an extension of $\theta$ to $B$ and $\psi \in \irrp{B}$.

By \cite[Lemma 25.5(a)]{H}, the set stabilizer 
$W_Y$ coincides with the inertia subgroup $I_W(\psi)$.
Hence, as each character $\beta_i$ is afforded by a real representation of $A_i$,
the  argument used in the proof of  \cite[Lemma 25.5(b)]{H} shows that 
$\psi$ has an extension $\hat{\psi} \in \irrp{W_Y}$.

Since $\theta = \psi_M$, we have $W_Y = I_W(\psi) \leq I_W(\theta)$. On the other hand, if $w \in I_W(\theta)$, then $w$ normalizes
$\ker \theta = \prod_{i \in X \smallsetminus Y}S_i$ and hence $w \in W_{X \smallsetminus Y} =  W_Y$. Therefore, $W_Y = I_W(\theta)$ and
$$G_Y = G \cap W_Y = G \cap I_W(\theta) = I_G(\theta).$$

Let  $\hat{\theta} = \hat{\psi}_{ G_Y}$ be the restriction of $\hat{\psi}$ to $G_Y$. 
As $\hat{\theta}_M = \hat{\psi}_M = \theta \in \irr M$,   $\hat{\theta}$ is an irreducible character of $G_Y$ and,
since $\hat{\psi} \in \irrp{W_Y}$, we  deduce that $\hat{\theta} \in \irrp{G_Y}$. 
  
Now, if $|G_Y/G_Y '|$ is even, then there exists a linear character $\lambda$ of $G_Y$
such that $o(\lambda) = 2$. Hence, by Gallagher's theorem and \cite[Lemma 4.8]{Is}   $\hat{\theta}$ and $\lambda\hat{\theta}$
are distinct characters in $\irrp{G_Y}$, and  by Clifford correspondence $\hat{\theta}^G$ and $(\lambda\hat{\theta})^G$ are distinct
characters of the same degree in $\irrp{G}$, against Hypothesis C. 

On the other hand, if  $|G_Y/G_Y'|\geq 5$, then  Clifford correspondence yields at least  five characters of the same degree in $\irr{G|\hat{\theta}}$,   a contradiction again.
Hence, $|G_Y: G_Y'| \in \{ 1,3\}$ and part (a) is proved.

\medskip

So far, we have shown that for every non-empty subset $Y$ of $X$ there exists a character $\chi_Y \in \irrp{G}$  such that
\begin{equation}\label{deg}
  \chi_Y(1) = \phi(1)^{|Y|}|G:G_Y| .
\end{equation}
Moreover, the irreducible constituents of the restriction of $\chi_Y$ to $M$ have kernels of the form
$$\prod_{i\in X\smallsetminus Y^g}S_i$$
for some  $g\in G$. Hence, if $Y_1, Y_2 \in X^{[k]}$ ($1 \leq k \leq n$) are sets belonging to different $G$-orbits, then
$\chi_{Y_1} \neq \chi_{Y_2}$. 

As  $|G:G_Y|$ is the size  of the orbit of $Y \in X^{[k]}$ in the action of $G$ on $X^{[k]}$, 
 part (b) follows immediately from (\ref{deg}) and  Hypothesis C.  
\end{proof}

\begin{thm}\label{non-solv-hypB}
  If $G$ is a non-solvable group  and $G$  satisfies Hypothesis C, then
$G$ is one of the groups in the list $\mathcal{L}$ defined in \eqref{list1}.
\end{thm}
\begin{proof}
  We work by induction on $|G|$. 

  \medskip
{\bf (I)}~~We first assume that there exists a non-solvable minimal normal subgroup $M$ of $G$, and we start by showing that  $\cent GM = 1$.

  Suppose, working by contradiction, that  $C = \cent GM  \neq 1$.
  Let $C_0$ be a minimal normal subgroup of $G$, with $C_0 \leq C$.
  As $G/C_0$ is non-solvable and  satisfies Hypothesis~C,  by induction  $G/C_0$ is a  group in $\mathcal{L}$.
  In particular, $G/C_0$ is almost-simple, so  $C = C_0$ is minimal normal in $G$,  $G/C$ has socle $MC/C \cong M$ and  $|G:MC| \in \{1,3\}$.

  Assuming  $G =  M \times C$, then  $C \cong G/M$ satisfies Hypothesis~C.
  If $C$ is solvable,  then $C  \cong C_3$   and hence, as $M$  (by direct inspection of the socles of the groups in $\mathcal{L}$ or by~\cite[Lemma 1]{BCH}) has
  at least two irreducible characters with the same degree,
  $G$ has at least six irreducible characters with the same degree, a contradiction.
  Hence, both $M$ and $C$ are simple groups belonging to the list $\mathcal{L}$. We observe that $M$ and $C$ are non-isomorphic,
 as otherwise $G$ has  irreducible characters with the same degree and Frobenius--Schur indicator, but with different kernels, a contradiction. 
  Every simple group in $\mathcal{L}$, except for
  ${\rm SL}_3(2)$, ${\rm M }_{22}$, ${\rm McL }$ and ${\rm Th}$, has three irreducible characters with the same degree and,
  as $G$ cannot have six characters of the same degree, we deduce
  that $M$ and $C$ are (non-isomorphic) groups belonging to the sublist $\{ {\rm SL}_3(2), {\rm M }_{22}, {\rm McL }, {\rm Th} \}$. 
  Now, one checks that for each of the six possibilities for $G = M \times C$,  $\irr G$ contains four non-real characters with the same degree,
  against Hypothesis C.

  Assuming instead that  $|G:MC| = 3$, then either $M \cong{\rm PSL}_2(8)$ or  $ M \cong  {\rm O}_8^+(2)$.
Hence,  there are three  characters $\phi_1, \phi_2, \phi_3 \in \irr M$ with the same degree, that are transitively
  permuted by the outer automorphism group of order $3$ of $M$.
  So, for every $\theta\in \irr C$, $I_G(\phi_i \times \theta) = MC$ and hence $(\phi_i \times \theta)^G \in \irr G$.
  If $M \cong{\rm SL}_2(8)$,   $\phi_i(1) = 7$ and $\phi_i$ is non-real, while if $M \cong  {\rm O}_8^+(2)$,  $\phi_i(1) = 35$ and $\phi_i$ is real, for $i = 1,2,3$.
  We observe that  $G/M$ satisfies  Hypothesis~C and that  $CM/M$ is a minimal normal subgroup of index three of $G/M$.
  If  $C$ is solvable, then by Theorem~\ref{solvable}
  either $G/M \cong C_7 \rtimes C_3$ and $C \cong C_7$,  or $G/M \cong \AAA_4$ and $C \cong C_2 \times C_2$.
If $C$ is non-solvable, then  by induction  either $C \cong{\rm PSL}_2(8)$ or  $ C \cong  {\rm O}_8^+(2)$.
In any case, there are at least three characters $\theta_1, \theta_2, \theta_3  \in \irr C$ having the same degree and either all non-real or all real.
It follows that the characters $\phi_i \times \theta_j$, for $1 \leq i,j \leq 3$, are
either nine non-real, or nine real, irreducible characters with the same degree of $MC$.
As $I_G(\phi_i \times \theta_j) = MC$, recalling~\cite[Lemma 2.1]{IN} we conclude that $G$ has 
either at least three non-real, or at least three real,  irreducible characters  with the same degree, against Hypothesis C.

\medskip
So far, we have proved that $\cent GM = 1$ and hence  $M$ is the only minimal normal subgroup of $G$.
Writing 
$$M = S_1 \times \cdots \times S_n,$$ 
with $S_i \simeq S$ isomorphic non-abelian simple groups, we define
$$N = \bigcap_{i= 1}^n \norm G{S_i}, \mbox{ and }H = G/N.$$
So,   $H$  is a transitive  permutation group on $X = \{1, \ldots, n\}$.
  If $n =  1$, then $G$ is an almost simple group and  $G \in \mathcal{L}$ by Theorem~\ref{as}.
 We now assume, aiming at a contradiction, that $n \neq 1$. We observe that then, in particular,  $H \neq 1$.  We have two cases. 

  \medskip
  {\bf (I.a)}~~We first assume that  $G/M$ is solvable. 
Hence, for every non-empty subset $Y$ of $X$, by  part (a) of Lemma~\ref{L1}  either $G_Y = M$ or $[G_Y:G_Y'] = 3$.  
By Theorem~\ref{solvable},
both $G/M$ and $H = G/N$
belong to the list
$$\begin{array}{rl}
    \mathcal{S} = & \bigl\{ C_3, C_7 \rtimes C_3, (C_2\times C_2 \times C_2) \rtimes (C_7 \rtimes C_3),
                       \AAA_4, \\
                  & \SL{3}, (C_3\times C_3) \rtimes \SL{3}, (C_5\times C_5) \rtimes \SL{3}   \bigr\}.
  \end{array}$$

  We remark that, as $H_Y$ is a quotient of $G_Y$,  
\begin{equation}
  \label{e1}
  \text{for all \ } Y \subseteq X, Y \neq \emptyset, \text{ \ either \ }  H_Y = 1 \text{ \  or \ } |H_Y:H_Y'| = 3 \, .
\end{equation}

Let $L= H_{x}$, for some $x \in X$,  be a point stabilizer in $H$; so $n = |X| = |H:L|$.
We first show that $H$ does not act regularly on $X$, i.e. that $L \neq 1$.
To prove this, we assume that $H$ acts regularly on $X$  and we consider a subset $Y = \{x_1, x_2\}\in X^{[2]}$.
For $h \in H_Y$, then $h^2$ fixes $x_i$, for $i= 1,2$, so $h^2 = 1$. It follows that $H_Y$ is an elementary abelian group, and hence $H_Y = 1$ by~(\ref{e1}).
  Thus, $H$ has $(n-1)/2$ orbits of size $n$ on $X^{[2]}$ and part (b) of Lemma~\ref{L1} yields  $n = 3$ and $H \cong C_3$. 
  Moreover, as $G/M \in \mathcal{S}$ and  $G_Y/M = N/M$, 
    part (a) of Lemma~\ref{L1} implies that $N = M$. 
   In particular, for every $\alpha \in \irr S$, $\alpha \times \alpha \times \alpha \in \irr M$ has three extensions to
  $\irr G$. Since $S$ has at least two irreducible characters with the same degree by~\cite[Lemma 1]{BCH},
  then $G$ has at least six irreducible characters with the same degree, a contradiction. 

Hence, $L \neq 1$ and, as  $H$ acts faithfully on $X$, the normal core $L_H$  of $L$ in $H$ is trivial. 
Recalling that  $L < H$ since $n \neq 1$ and that   $|L/L'| = 3$ by~(\ref{e1}), we conclude that 

$$ |X| \in N(H) = \bigl\{|H:L| \mid 1 \neq L < H, |L/L'| = 3 \text{ and } L_H = 1\bigr\}.$$
In order to reduce the number of cases for the relevant  transitive actions of the groups in $\mathcal{S}$, we  introduce the
set
$$W(H) = \bigl\{|H:T| \mid T < H \text{ and } |T/T'| = 3 \} \cup \{|H|\bigr\}$$
and we claim that, for every   $0 < k < n$,  the size $|H:H_Y|$ of the $H$-orbit of $Y \in  X^{[k]}$ belongs to $W(H)$. 
In fact, $H_Y < H$ as $H$ is transitive on $X$ and $Y \neq X$, so the claim follows by~(\ref{e1}).

Thus, recalling that $n =|X| \in N(H)$,  part (b) of Lemma~\ref{L1} yields that, for every $k$ such that  $0 < k < n$,
\begin{equation}\label{e2}
   |X^{[k]}| = {n \choose k} \text{ is a  sum of distinct elements of } W(H) .  
\end{equation}

Since $H \not\cong C_3$ as $H$ acts non-regularly on $X$, the triples $(H, N(H), W(H))$ are listed in Table~\ref{T.1a}.

\begin{table} 
\begin{center}
\small
\[
\begin{array}{|l|r|r|r|r|r|r|} \hline
H   & C_7 \rtimes C_3  &  C_2^3 \rtimes (C_7 \rtimes C_3) & \AAA_4 &  \SL{3}   & C_3^2 \rtimes \SL{3} & C_5^2 \rtimes \SL{3} \\
\hline
  N(H) & \{7\} & \{8, 14, 56\} & \{ 4\}& \{8\} & \{ 9, 72\}& \{ 25, 200\}  \\
  \hline
    W(H) &\{ 7, 21\} & \{ 8, 14, 56, 168\} &\{ 4, 12\} & \{8, 24\} &\{ 9, 72, 216\} & \{25, 200, 600\} \\
                   \hline
\end{array}
\]
\caption{}
\label{T.1a}
\end{center}        
\end{table}

One easily checks that (\ref{e2}) is not satisfied for $k=3$ if $H\cong C_7 \rtimes C_3$, while it is not satisfied for $k=2$
in all other cases, a contradiction. 

    \bigskip
    {\bf (I.b)}~~We assume now that  $G/M$ is non-solvable.
    Since $G/M$ satisfies Hypothesis C,  by induction  $G/M \in \mathcal{L}$.
    We observe that $N/M$, being isomorphic to a subgroup of a direct product of outer automorphism groups of a simple group,
    is a solvable normal subgroup of $G/M$. Hence, $N = M$ and $H = G/M$.

  In order to reduce the number of cases for the action of $H$ on $X$, we consider the function
  $$b(|H|) = \max  \Bigl\{ a \in d(|H|) \mid \sum_{1 \neq d \in d(|H|)}d \geq {a \choose \lfloor\frac{a}{ 2} \rfloor } \Bigr\}$$
  where $d(|H|)$ is the set of the positive divisors of  $|H|$.
  By the orbit formula and part (b) of Lemma~\ref{L1}, setting $m = \lfloor\frac{|X|}{2} \rfloor$,  $|X^{[m]}|$  is a sum of
  distinct elements (orbit sizes)  $d_i \in d(|H|)$. Moreover, every $d_i \neq 1$, because $H$ has no fixed points in $X^{[m]}$
  since  $H$ is transitive on $X$ and $0 < m < |X|$.

  We deduce that $|X|  \leq b(|H|)$ and hence  $|X|$ belongs to the  set
  $$N_1(H) = \bigl\{ c \cdot |H:L| \mid c \in \mathbb{Z}_{+}, L \text{ is a maximal subgroup of } H \text{ and } c \cdot |H:L| \leq b(|H|) \bigr\}.$$

  We remark that, for $H \in \mathcal{L}$, both $b(|H|)$ and $N_1(H)$ can be computed using the information in~\cite{atlas} and that
  the set stabilizers $H_Y$ for the relevant actions of $H$ on $X^{[k]}$, for $|X| \in N_1(H)$ and $0<k<|X|$,  can be computed in~\cite{GAP}.
  In particular, we get that $N_1(H) = \emptyset$ for 
 $$H \in \bigl\{ {\rm SU}_3(3), {\rm McL}, {\rm Th}, {\rm O}_8^+(2).3 \bigr\}.$$
  For the remaining cases, we refer to Table~\ref{T.1b}.

\begin{table} 
\begin{center}
\small
\[
\begin{array}{|l|r|r|r|r|r|r|r|} \hline
H   &  \AAA_8 &  {\rm PSL}_3(2)    & {\rm M}_{11} &  {\rm M}_{22}   & {\rm M}_{23} & {\rm M}_{24} & {\rm PSL}_2(8).3 \\
\hline
  b(|H|) & 18 & 8 & 16& 22 & 28 & 32 & 14  \\
  \hline
    N_1(H) &\{ 8, 15\} & \{ 7,8\} &\{ 11, 12\} & \{22 \} & \{ 23\} &\{ 24\} & \{ 9 \}   \\
                   \hline
\end{array}
\]
\caption{}
\label{T.1b}
\end{center}        
\end{table}

For $(H, |X|) \in \{ ({\rm PSL}_3(2), 7), ({\rm M}_{11}, 11), ({\rm M}_{11}, 12)   \}$ and $Y = \{x \} \subseteq X$,
$H_Y$ has a quotient of order $2$, a contradiction by part (a) of Lemma~\ref{L1}.
In all other cases, for $Y \in X^{[2]}$, $H_Y$ has a quotient of order $2$, again a contradiction. 
  
  \bigskip
  ({\bf II.})
  We can now  assume   that $G$ has no non-solvable minimal normal subgroups.
  So,  $M$ is solvable and  $G/M$ is a non-solvable group that satisfies Hypothesis C.
  Hence, by induction $G/M$ is an almost-simple group belonging to the list  $\mathcal{L}$. Let $K/M$ be the socle of $G/M$.

  We observe that $M \leq K'$. In fact, $K'M/M = (K/M)' = K/M$, hence $K = K'M$,  and $K' \nor G$, so  if $M \not\leq K'$, then
  $M \cap K' = 1$ and $K' \cong K/M$ is a non-solvable minimal normal subgroup of $G$, a contradiction.   

  If $M < \cent GM$, then $M  \leq \zent K$ and hence $K$ is a quotient of the Schur cover of $K/M$.
If $G = K$, i.e.  $G/M$ is a simple group in $\mathcal{L}$,  by direct inspection using~\cite{atlas}, we see that all the groups arising in this way do not satisfy Hypothesis C.  
Since the Schur multiplier of ${\rm SL}_2(8)$ is trivial, the only case left is $G/M \cong {\rm O}_8^+(2).3 $.
Hence,  $|M| = 2^2$
(as the non-trivial elements of the Schur multiplier of ${\rm O}_8^+(2)$ are transitively permuted by the subgroup of order three of
its outer automorphism group) and $G$ has three real irreducible characters of degree $3024$, a contradiction.

Therefore, $M$ is the unique minimal normal subgroup of $G$ and it  is a faithful non-trivial irreducible $G/M$-module.
Hence, an application of  Theorem~\ref{modules} gives the final contradiction, concluding the proof.
  \end{proof}

Now Theorem~A follows immediately from Theorem~\ref{solvable} and Theorem~\ref{non-solv-hypB}.

\bibliographystyle{amsalpha}
\bibliography{bcdnt_revised}

\end{document}